\documentclass[a4paper,12pt,final]{amsart}
\usepackage{times,a4wide,mathrsfs,amssymb}

\newcommand{\C}{\mathbb{C}}
\newcommand{\ZZ}{\mathbb{Z}}

\newcommand{\QQ}{\mathbb{Q}}
\newcommand{\NN}{\mathbb{N}}
\newcommand{\PP}{\mathbb{P}}

\newcommand{\OO}{\mathcal O}

\newcommand{\MM}{\mathcal M}

\newcommand{\wt}{\widetilde}

\newcommand{\rom}{\romannumeral}
\newcommand{\alb}{\hbox{Alb}}

\newtheorem{theorem}{Theorem}[section]

\newtheorem{lemma}[theorem]{Lemma}
\newtheorem{corollary}[theorem]{Corollary}
\newtheorem{proposition}[theorem]{Proposition}
\newtheorem{conjecture}[theorem]{Conjecture}
\newtheorem{remark}[theorem]{Remark}
\newtheorem{definition}[theorem]{Definition}
\newtheorem{convention}{Conventions}

\newtheorem{notation}[theorem]{Notation}

\newtheorem{nonumbering}{Theorem}

\newtheorem{nonumberingt}{Acknowledgements}

\begin{document}
\author[Robert Laterveer]
{Robert Laterveer}

\address{Institut de Recherche Math\'ematique Avanc\'ee,
CNRS -- Universit\'e 
de Strasbourg,\
7 Rue Ren\'e Des\-car\-tes, 67084 Strasbourg CEDEX,
FRANCE.}
\email{robert.laterveer@math.unistra.fr}

\title{Bloch's conjecture for Enriques varieties}

\begin{abstract} Enriques varieties have been defined as higher--dimensional generalizations of Enriques surfaces. Bloch's conjecture implies that Enriques varieties should have trivial Chow group of zero--cycles. We prove this is the case for all known examples of irreducible Enriques varieties of index larger than $2$. The proof is based on results concerning the Chow motive of generalized Kummer varieties.
\end{abstract}



\maketitle

\section{Introduction}

For a smooth complex projective variety $X$, let $A_jX$ denote the Chow group of dimension $j$ algebraic cycles on $X$ modulo rational equivalence. Let $A_j^{hom}(X)\subset A_j(X)$ denote the subgroup of homologically trivial cycles.
Other than the case of divisors ($j=\dim X-1$), Chow groups are in general still poorly understood. For example, there is the famous conjecture of Bloch:

\begin{conjecture}[Bloch \cite{B}]\label{bloch} Let $X$ be a smooth projective complex variety. The following are equivalent:

\noindent
(\rom1) the Albanese morphism $A^{hom}_0{}(X)\to\alb(X)$ is an isomorphism;

\noindent
(\rom2) the Hodge numbers $h^{j,0}(X)$ are $0$ for $j\ge 2$.
\end{conjecture}

The implication from (\rom1) to (\rom2) is actually a theorem \cite{Mum}, \cite{BS}. The conjectural part is the reverse implication, which has been verified for surfaces of Kodaira dimension less than $2$ \cite{BKS}, but is wide open for surfaces of general type (cf. \cite{PW}, \cite{V8} for some examples of surfaces where conjecture \ref{bloch} is verified).

Interesting examples of varieties with vanishing Hodge numbers $h^{j,0}(X)=0$ for all $j\ge 1$ are given by {\em Enriques varieties\/}. These varieties have been defined and studied by Boissi\`ere, Nieper-Wi{\ss}kirchen and Sarti in \cite{BNS} (and independently, with a somewhat different definition, by Oguiso--Schr\"oer in \cite{OS}). As the name suggests, Enriques varieties are higher--dimensional generalizations of Enriques surfaces. In the same way that Enriques surfaces are closely related to $K3$ surfaces, the study of Enriques varieties is intimately entwined with that of hyperk\"ahler varieties. By definition, an Enriques variety $X$ has the property that some multiple $dK_X$ of the canonical divisor is trivial; the smallest such positive integer $d$ is called the {\em index\/} of $X$.

It is natural to ask whether one can prove Bloch's conjecture for these varieties, i.e.

\begin{conjecture}\label{en} Let $X$ be an Enriques variety (in the sense of \cite{BNS}). Then
  \[ A_0(X)=\ZZ\ .\]
  \end{conjecture}
  
 The main result of this note gives a partial answer to conjecture \ref{en}:
 
 \begin{nonumbering}[=theorem \ref{main}]  Let $X$ be an Enriques variety of dimension $\le 6$. Assume $X$ is a quotient 
  \[ X=K/G\ ,\]
  where $K=K_n(A)$ is a generalized Kummer variety and $G$ is a group of automorphisms acting freely and induced by a finite order automorphism of $A$.
  Then
    \[ A_0(X)=\ZZ\ .\]
   \end{nonumbering} 

Theorem \ref{main} applies to all known examples of irreducible Enriques varieties with index $>2$ (these examples can be found in \cite{BNS} and \cite{OS}). The proof of theorem \ref{main} is a straightforward application of results of Xu \cite{Xu} and Lin \cite{Lin}, combined with Kimura's theory of {\em finite--dimensional 
motives\/} \cite{Kim}.

As a corollary (corollary \ref{BS}), varieties as in theorem \ref{main} verify certain cases of the generalized Hodge conjecture.

\vskip0.6cm

\begin{convention} In this note, the word {\sl variety\/} will refer to a reduced irreducible scheme of finite type over $\C$. 

For any variety $X$, we will denote by $A_jX$ the Chow group of $j$--dimensional cycles on $X$, and we will write
  \[ A_j(X)_{\QQ}:= A_j(X)\otimes_{\ZZ} \QQ \]
  for Chow groups
with rational coefficients.
For $X$ smooth of dimension $n$ the notations $A_jX$ and $A^{n-j}X$ will be used interchangeably. 

The notations 
$A^j_{hom}(X)$ and $A^j_{AJ}(X)$ will be used to indicate the subgroups of 
homologically, resp. Abel--Jacobi trivial cycles.
The contravariant category of Chow motives (i.e., pure motives with respect to rational equivalence as in \cite{Sc}, \cite{MNP}) will be denoted $\MM_{\rm rat}$. The category of pure motives with respect to homological equivalence will be denoted $\MM_{\rm hom}$.
\end{convention}

\section{Preliminary material}

\subsection{Quotient varieties}

\begin{definition} A {\em projective quotient variety\/} is a variety
  \[ X=Y/G\ ,\]
  where $Y$ is a smooth projective variety and $G\subset\hbox{Aut}(Y)$ is a finite group.
  \end{definition}
  
 \begin{proposition}[Fulton \cite{F}]\label{quot} Let $X$ be a projective quotient variety of dimension $n$. Let $A^\ast(X)$ denote the operational Chow cohomology ring. The natural map
   \[ A^i(X)_{\QQ}\ \to\ A_{n-i}(X)_{\QQ} \]
   is an isomorphism for all $i$.
   \end{proposition}
   
   \begin{proof} This is \cite[Example 17.4.10]{F}.
      \end{proof}

\begin{remark} It follows from proposition \ref{quot} that the formalism of correspondences goes through unchanged for projective quotient varieties (this is also noted in \cite[Example 16.1.13]{F}). We can thus consider motives $(X,p,0)\in\MM_{\rm rat}$, where $X$ is a projective quotient variety and $p\in A^n(X\times X)_{\QQ}$ is a projector. For a projective quotient variety $X=Y/G$, one readily proves (using Manin's identity principle) that there is an isomorphism
  \[  h(X)\cong h(Y)^G:=(Y,\Delta^G_Y,0)\ \ \ \hbox{in}\ \MM_{\rm rat}\ ,\]
  where $\Delta^G_Y$ denotes the idempotent ${1\over \vert G\vert}{\sum_{g\in G}}\Gamma_g$.  
  \end{remark}

\subsection{Finite--dimensionality}

We refer to \cite{Kim}, \cite{An}, \cite{MNP}, \cite{J4} for basics on the notion of finite--dimensional motive. 
An essential property of varieties with finite--dimensional motive is embodied by the nilpotence theorem:

\begin{theorem}[Kimura \cite{Kim}]\label{nilp} Let $X$ be a smooth projective variety of dimension $n$ with finite--dimensional motive. Let $\Gamma\in A^n(X\times X)_{\QQ}$ be a correspondence which is numerically trivial. Then there is $N\in\NN$ such that
     \[ \Gamma^{\circ N}=0\ \ \ \ \in A^n(X\times X)_{\QQ}\ .\]
\end{theorem}

 Actually, the nilpotence property (for all powers of $X$) could serve as an alternative definition of finite--dimensional motive, as shown by a result of Jannsen \cite[Corollary 3.9]{J4}.
   Conjecturally, all smooth projective varieties have finite--dimensional motive \cite{Kim}. We are still far from knowing this, but at least there are quite a few non--trivial examples:

\begin{remark}
It is an embarassing fact that up till now, all examples of finite-dimensional motives happen to lie in the tensor subcategory generated by Chow motives of curves, i.e. they are ``motives of abelian type'' in the sense of \cite{V3}. On the other hand, there exist many motives that lie outside this subcategory, e.g. the motive of a very general quintic hypersurface in $\PP^3$ \cite[7.6]{D}.
\end{remark}

The notion of finite--dimensionality is easily extended to quotient varieties:

\begin{definition} Let $X=Y/G$ be a projective quotient variety. We say that $X$ has finite--dimensional motive if the motive
  \[ h(Y)^G:= (Y, \Delta^G_Y,0)\ \ \ \in \MM_{\rm rat}\]
  is finite--dimensional. (Here, $\Delta^G_Y$ denotes the idempotent ${1\over \vert G\vert}{\sum_{g\in G}}\Gamma_g \in A^n(Y\times Y)$.) 
   \end{definition}
  
  Clearly, if $Y$ has finite--dimensional motive then also $X=Y/G$ has finite--dimensional motive. The nilpotence theorem extends to this set--up:
  
  \begin{proposition}\label{quotientnilp} Let $X=Y/G$ be a projective quotient variety of dimension $n$, and assume $X$ has finite--dimensional motive. Let $\Gamma\in A^n_{num}(X\times X)_{\QQ}$. Then there is   
    $N\in\NN$ such that
     \[ \Gamma^{\circ N}=0\ \ \ \ \in A^n(X\times X)_{\QQ}\ .\]
  \end{proposition}
  
  \begin{proof} Let $p\colon Y\to X$ denote the quotient morphism.
     We associate to $\Gamma$ a correspondence $\Gamma_Y\in A^n(Y\times Y)_{\QQ}$ defined as
    \[ \Gamma_Y:= {}^t \Gamma_p\circ \Gamma\circ \Gamma_p\ \ \ \in A^n(Y\times Y)_{\QQ}\ .\]
  By Lieberman's lemma \cite[Lemma 3.3]{V3}, there is equality
   \[ \Gamma_Y =(p\times p)^\ast \Gamma\ \ \ \hbox{in}\ A^n(Y\times Y)_{\QQ}\ ,\]
   and so $\Gamma_Y$ is $G\times G$--invariant:
   \[ \Delta_Y^G\circ \Gamma_Y\circ \Delta_Y^G =\Gamma_Y\ \ \ \hbox{in}\ A^n(Y\times Y)_{\QQ}\ .\]
   This implies that 
     \[\Gamma_Y\in \Delta_Y^G\circ A^n(Y\times Y)_{\QQ}\circ\Delta_Y^G\ ,\] 
     and so
   \[ \Gamma_Y\in\hbox{End}_{\MM_{\rm rat}}\bigl(h(Y)^G\bigr)\ .\]
   Since clearly $\Gamma_Y$ is numerically trivial, and $h(Y)^G$ is finite--dimensional (by assumption), there exists $N\in\NN$ such that
    \[ (\Gamma_Y)^{\circ N} = {}^t \Gamma_p\circ \Gamma\circ\Gamma_p\circ{}^t \Gamma_p\circ \cdots \circ \Gamma_p=0\ \ \ \hbox{in}\ A^n(Y\times Y)_{\QQ}\ .\]
    Using the relation $\Gamma_p\circ{}^t \Gamma_p=d\Delta_X$, this boils down to
    \[ d^{N-1}\ \  {}^t \Gamma_p\circ \Gamma^{\circ N}\circ \Gamma_p=0\ \ \ \hbox{in}\ A^n(Y\times Y)_{\QQ}\ .\]
    From this, we deduce that also
    \[ \Gamma^{\circ N}= {1\over d^{N+1}} \Gamma_p\circ \Bigl( d^{N-1} \ \ {}^t \Gamma_p\circ \Gamma^{\circ N}\circ \Gamma_p\Bigr) \circ {}^t \Gamma_p=0\ \ \ \hbox{in}\ A^n(X\times X)_{\QQ}\ .\]    
     \end{proof}

\subsection{Enriques varieties}

\begin{definition}[\cite{BNS}] A smooth projective variety is called {\em Enriques variety\/} if the following hold:

\noindent
(\rom1) the holomorphic Euler characteristic $\chi(X,\OO_X)=1$;

\noindent
(\rom2)
there exists an integer $d\ge 2$ (called the {\em index\/} of $X$) such that the canonical bundle $K_X$ has order $d$ in the Picard group of $X$, and the fundamental group
$\pi_1(X)$ is cyclic of order $d$.
\end{definition}

\begin{definition}[\cite{BNS}] An Enriques variety is called {\em irreducible\/} if the holonomy representation of its universal cover is irreducible.
\end{definition}

\begin{theorem}[\cite{BNS}] Let $X$ be an irreducible Enriques variety of index $>2$. Then $X$ is the quotient of an irreducible symplectic holomorphic manifold by a group of automorphisms acting freely.
\end{theorem}

\begin{proposition}[\cite{BNS}]\label{ex} There exist irreducible Enriques varieties of dimension $4$ and index $3$, and of dimension $6$ and index $4$.
\end{proposition}

\begin{proof} This is \cite[Proposition 4.1]{BNS}, the idea of which is as follows. Let $A$ be the product of 2 elliptic curves, and let $\phi$ be a finite order automorphism of $A$. Consider the generalized Kummer variety $K=K_n(A)$ for $n=3,4$. For an appropriate choice of $\phi$, the induced automorphism $\psi\in\hbox{Aut}(K)$ is such that the action on $K$ is free, and the quotient
  \[ X=K/<\psi>\ \]
  is an Enriques variety.
\end{proof}

\begin{remark} To the best of my knowledge, there are as yet no examples of Enriques varieties with index $>4$.
\end{remark}

\begin{remark} In \cite{OS}, there is a definition of ``Enriques manifold'' which is a priori slightly different from the definition of Enriques variety. (In \cite[Remark 1.3(a)]{JKim}, it is explained there might potentially exist Enriques varieties that are {\em not\/} Enriques manifolds.)
However, the examples given in proposition \ref{ex} are also Enriques manifolds (and actually, these examples are also to be found in \cite{OS}).
\end{remark}

\subsection{Generalized Kummer varieties}

\begin{definition} Let $A$ be an abelian surface. For any $n\in\NN$, let 
  \[    \pi\colon A^{[n]}\ \to\ A^{(n)}\]
  denote the Hilbert--Chow morphism from the Hilbert scheme $A^{[n]}$ to the symmetric product $A^{(n)}$. Let $\sigma\colon A^{(n)}\to A$ denote the addition morphism. Consider the composition
  \[ s\colon A^{[n]}\ \xrightarrow{\pi}\ A^{(n)}\ \xrightarrow{\sigma}\ A\ .\]
  The generalized Kummer variety is defined as the fibre
  \[  K_n(A):=s^{-1}(0)\ .\]
$K_n(A)$ is a hyperk\"ahler variety of dimension $2n-2$.
\end{definition}

\begin{definition}[\cite{BNS}]\label{nat} An automorphism $\psi\in\hbox{Aut}(K_n(A))$ is {\em natural\/} if $\psi$ is induced by an automorphism of $A$. More precisely, let $A[n]$ denote the $n$--torsion points of $A$, and let $\hbox{Aut}_{\ZZ}(A)$ denote the group automorphisms of $A$. As explained in \cite[Section 3.1]{BNS}, there is a well--defined homomorphism
  \[ A[n]\rtimes \hbox{Aut}_\ZZ(A)\ \to\ \hbox{Aut}(K_n(A))\ .\]
  The group of natural automorphisms of $K_n(A)$ is defined as the image of this homomorphism.
\end{definition}

\begin{theorem}[Boissi\`ere--Nieper-Wi{\ss}kirchen--Sarti \cite{BNS}] Let $n\ge 3$. Let $E$ denote the exceptional divisor of the birational morphism (obtained from $\pi$ by restriction)
  \[ \pi\vert_{K_n(A)}\colon\ \ K_n(A)\ \to\ K^{(n)}:=\sigma^{-1}(0)\ .\]
  An automorphism $\psi\in\hbox{Aut}(K_n(A))$ is natural if and only if $\psi(E)=E$.
  \end{theorem}
  
  \begin{proof} This is  \cite[Theorem 3.1]{BNS}.  
   \end{proof}

\subsection{Motive of a generalized Kummer variety}
\label{mot}

\begin{notation} For $n\in\NN$, let $P(n)$ be the set of partitions of $n$. A partition $\lambda\in P(n)$ can be written
  \[ \lambda= (\lambda_1,\ldots,\lambda_{\ell_\lambda})=1^{a_1}2^{a_2}\cdots r^{a_r}\ ,\]
  where $\ell_\lambda$ is the length of $\lambda$.
  We define $e(\lambda):=\hbox{gcd}\{\lambda_1,\ldots,\lambda_{\ell_\lambda}\}$.
  
  For any $\lambda\in P(n)$, we write
  \[ A^{(\lambda)}:= A^{(a_1)}\times A^{(a_2)}\times\cdots\times A^{(a_r)}\ .\] 
\end{notation}

\begin{definition} A homomorphism of Chow motives $\Gamma\colon M\to N$ is {\em split\/} if $\Gamma$ admits a left inverse.
\end{definition}

\begin{theorem}[Xu \cite{Xu}]\label{xu} Let $K=K_n(A)$ be a generalized Kummer variety. There is a split homomorphism of Chow motives
  \[ \Gamma\colon\ \ h(K)\ \to\  \bigoplus_{\lambda\in P(n)} \bigoplus_{\tau\in A[e(\lambda)]} h(A^{(\lambda)})(n-\ell_\lambda-2)\ \ \ \hbox{in}\ \MM_{\rm rat}\ .\]
  In particular, $K$ has finite--dimensional motive, in the sense of \cite{Kim} (and even: $K$ has motive of abelian type, in the sense of \cite{V3}).
\end{theorem}

\begin{proof} This follows from \cite[Corollary 2.8]{Xu}, which states more precisely that there is an isomorphism
  \[  h(A\times K)\ \cong\ \bigoplus_{\lambda\in P(n)} \bigoplus_{\tau\in A[e(\lambda)]} h(A^{(\lambda)})(n-\ell_\lambda)\ \ \ \hbox{in}\ \MM_{\rm rat}\ .\]
Theorem \ref{xu} is obtained by composing with the split homomorphism
\[ h(K)\ \to\ h(A\times K)\ \to\ h(A\times K)(-2)  \ ,   \]
where the first arrow is given by projection on the second summand, and the second arrow is given by intersecting with $x\times K$ where $x\in A$.
\end{proof}

\begin{remark} The fact that generalized Kummer varieties have finite--dimensional motive of abelian type (which was first stated explicitly in \cite{Xu}) seems to have been folklore knowledge for quite some time. Indeed, as noted in \cite[Remark 7.10 and \S 6.1]{FTV}, this fact follows readily from the results of de Cataldo--Migliorini \cite{CM}.

We mention in passing that L. Fu, in the course of proving the Beauville--Voisin conjecture for generalized Kummer varieties, had previously developed a motivic decomposition for $K_n(A)$ \cite{LFu}. Parts of the argument of Xu \cite{Xu} can already be found in \cite{LFu}. 
\end{remark}

Things simplify if one is only interested in zero--cycles:

\begin{corollary}\label{corxu} Let $K=K_n(A)$ be a generalized Kummer variety. There is a split injection
  \[ \Gamma_\ast\colon\ \ A_0(K)_{\QQ}\ \to\ A_0(A^{(n)})_{\QQ}\ .\]
  \end{corollary}

\begin{proof} This is a consequence of theorem \ref{xu}; all summands with $\ell_\lambda<n$ vanish for dimension reasons.
\end{proof}

\begin{theorem}[Lin \cite{Lin}]\label{lin} Let $K=K_n(A)$ be a generalized Kummer variety. There exists a Chow--K\"unneth decomposition for $K$, i.e. a set of mutually orthogonal idempotents $\Pi^K_0,\ldots,\Pi^K_{4n-4}$ in $A^{2n-2}(K\times K)_{\QQ}$ lifting the K\"unneth components. 
Moreover, this decomposition satisfies
  \[  \bigl( \Pi^K_0+\Pi^K_2+\Pi^K_4+\cdots +\Pi^K_{4n-4}\bigr){}_\ast =\hbox{id}\colon\ \ A_0(K)_{\QQ}\ \to\ A_0(K)_{\QQ}\ .\]
  \end{theorem}
  
  \begin{proof} This is essentially \cite[Proposition 4.5]{Lin}. It follows from theorem \ref{xu} that $K$ is motivated by $A$ (in the sense of Arapura \cite{A}). Thus, \cite[Lemma 4.2]{A} implies $K$ verifies the standard Lefschetz conjecture, and so in particular the K\"unneth components of $K$ are algebraic. Finite--dimensionality then gives a Chow--K\"unneth decomposition \cite[Lemma 5.4]{J}. As for the last statement, this follows from the fact that the Beauville filtration on Chow groups of abelian varieties induces a decomposition
    \[  A_0(K)_{\QQ}=\bigoplus_{j=0}^{2n-2} A_0^{(j)}(K)_{\QQ}\ \]
    such that
    \[  A_0^{(j)}(K)_{\QQ}=\begin{cases} (\Pi^K_{4n-4-j})_\ast A_0(K)_{\QQ} &\hbox{if}\ j\ {is\ even};\\
               0 &\hbox{if}\ j\ {is\ odd}\ \\
               \end{cases}\]
   \cite[Theorem 1.4]{Lin}.   
  \end{proof}

Using the existence of a Chow--K\"unneth decomposition, corollary \ref{corxu} can be made more precise:

\begin{corollary}\label{corxu2} Let $K=K_n(A)$ be a generalized Kummer variety. Let $\Pi_j^K$ (resp. $\Pi_j^{A^{(n)}}$) be any Chow--K\"unneth decomposition of $K$ (resp. of $A^{(n)}$). For any $j$, there are split injections
  \[ \Gamma_\ast\colon\ \ 
       (\Pi^K_j)_\ast A_0(K)_{\QQ}\ \to\ (\Pi_{j+4}^{A^{(n)}})_\ast A_0(A^{(n)})_{\QQ}\ .\]
  \end{corollary}

\begin{proof} Let $\Psi$ denote a left inverse to the homomorphism 
  \[    \Gamma\colon h(K)\ \to\ N\ \ \ \hbox{in}\ \MM_{\rm rat}\ ,\]
  where $N\in \MM_{\rm rat}$ is a short--hand for the right--hand side of theorem \ref{xu}.
  As a consequence of theorem \ref{xu}, there are decompositions
  \[  \begin{split} &\Gamma=\Gamma_0+\Gamma_1\colon\ h(K)\ \to\ h(A^{(n)})\oplus N_1\ ,\\
                               &\Psi=(\Psi_0,\Psi_1)\colon\ h(A^{(n)})\oplus N_1\ \to\ h(K)\ \ \ \hbox{in}\ \MM_{\rm rat}\ \\
                               \end{split}\]
         satisfying
         \[       \Psi\circ \Gamma=       \Psi_0\circ \Gamma_0 +  \Psi_1\circ \Gamma_1=\hbox{id}\colon\ \ h(K)\ \to\ h(K)\ \ \ \hbox{in}\ \MM_{\rm rat}\ .\]                             
          Since $\Gamma_0$ sends $H^j(K,\QQ)$ to $H^{j+4}(A^{(n)},\QQ)$, there is a homological equivalence
          \[  L:=\Pi_j^K \circ\Psi_0\circ \Gamma_0   \circ \Pi_j^K     =\Pi_j^K \circ\Psi_0\circ \Pi_{j+4}^{A^{(n)}}\circ \Gamma_0   \circ \Pi_j^K =:R\ \ \ \hbox{in}\ H^{4n-4}(K\times K,\QQ)\ .\]
    Since $K$ has finite--dimensional motive, this means the difference $L-R$ is nilpotent. Upon developing, this implies
    \[ L^{\circ N} = Q_1+Q_2+\cdots+Q_N \ \ \ \hbox{in}\ A^{2n-2}(K\times K)_{\QQ}\ ,\]
   where $Q_j$ is a composition of $L$ and $R$ in which $R$ occurs at least once.
    
    Applying this to $0$--cycles, we obtain in particular
    \[   (L^{\circ N}){}_\ast= \bigl( Q_1+\cdots +Q_N\bigr){}_\ast\colon\ \ \ A_0(K)_{\QQ}\ \to\ A_0(K)_{\QQ}\ .\]
    Now we note that
      \[ L_\ast= (\Pi_j^K)_\ast\colon\ \ A_0(K)_{\QQ}\ \to\ A_0(K)_{\QQ}\] 
      thanks to corollary \ref{corxu}.
      It follows that
      \[  \begin{split}(\Pi_j^K)_\ast=((\Pi_j^K)^{\circ N})_\ast =(L^{\circ N})_\ast   & = \bigl(Q_1+\cdots +Q_N\bigr){}_\ast\\
      &=\bigl((\hbox{something})\circ R\circ \Pi_j^K\bigr){}_\ast\\
                                               &=\bigl( (\hbox{something})\circ \Pi_{j+4}^{A^{(n)}}\circ \Gamma_0\circ \Pi_j^K\bigr){}_\ast\colon\ \ \ A_0(K)_{\QQ}\ \to\ A_0(K)_{\QQ}\ .\\
                                               \end{split}
                                               \]
       It follows that
       \[ \hbox{id}=  \bigl( (\hbox{something})\circ \Pi_{j+4}^{A^{(n)}}\circ \Gamma_0\bigr){}_\ast\colon\ \  (\Pi^K_j)_\ast A_0(K)_{\QQ}\ \to\ (\Pi^K_j)_\ast A_0(K)_{\QQ}\ .\]
    This proves corollary \ref{corxu2}.
    \end{proof}

\section{Main result} 

\begin{theorem}\label{main} Let $X$ be an Enriques variety that is a quotient 
  \[ X=K/G\ ,\]
  where $K=K_n(A)$ is a generalized Kummer variety with $n\le 4$, and $G$ is a group of automorphisms acting freely and induced by a finite order automorphism of $A$.
  Then
    \[ A_0(X)=\ZZ\ .\]
   \end{theorem} 

\begin{proof} The theorem is true for $n=2$, so we will suppose from now on that $n\ge 3$.
Thanks to Rojtman \cite{R}, we only need to prove that $A_0(X)_{\QQ}=\QQ$.

Write $G=<\psi>$ where $\psi\in\hbox{Aut}(K)$ is an automorphism (of order $d=\hbox{index}(X)$) induced by a finite order automorphism 
    \[\phi\in A[n]\rtimes \hbox{Aut}_\ZZ(A)\subset \hbox{Aut}(A)\ .\] 
    We will write $\phi=t\circ \phi_0$, where $t$ is a translation on $A$ and $\phi_0$ is a group automorphism. Let
  \[ A^\prime:= A/<\phi_0> \ .\]
  The surface $A^\prime$ has at most quotient singularities (note that $\phi$ and $\phi_0$ might well have fixpoints even though $\psi$ is fixpoint free). The action of $\phi_0$ must be non--symplectic (for otherwise $p_g(X)=1$), and so $p_g(A^\prime)=0$. 
     
  We have seen that the K\"unneth components of $X$ are algebraic (this follows from theorem \ref{xu}, or from the results of \cite{CM}). Combined with the fact that $X$ has finite--dimensional motive, this implies \cite[Lemma 5.4]{J} that
  there exists a Chow--K\"unneth decomposition $\Pi_j^X\in A^{2n-2}(X\times X)_{\QQ}$ for $X$. 
  To prove theorem \ref{main}, it suffices to prove 
  \begin{equation}\label{vanish}  (\Pi^X_j)_\ast A_0^{hom}(X)_{\QQ}=0\ \ \ \hbox{for\ all\ }j\ .\end{equation}
  
  The next lemma enables us to change the Chow--K\"unneth projectors to our convenience; we are not stuck with one particular Chow--K\"unneth decomposition.
  
  \begin{lemma}\label{jans} Let $X$ be a variety with finite--dimensional motive, and such that the K\"unneth components of $X$ are algebraic. Let $\Pi_j^X$ and $\hat{\Pi}_j^X$ be two Chow--K\"unneth decompositions for $X$. Then for any $i$ and $j$, there is equivalence
    \[ (\Pi^X_j)_\ast A_i^{hom}(X)_{\QQ}=0  \ \Leftrightarrow\  (\hat{\Pi}^X_j)_\ast A_i^{hom}(X)_{\QQ}=0  \ .\]
    \end{lemma}
    
    \begin{proof} This is well--known, and easily proven. For later use, we prove a slightly more general statement:
    
   \begin{lemma}\label{jans2} Let $X$ be as in lemma \ref{jans}. Let $\Pi_j^X$ be a Chow--K\"unneth decomposition, and let $\hat{\pi}_j^X$ be any (not necessarily idempotent, or orthogonal) cycles mapping to the K\"unneth components $\pi_j\in H^\ast(X\times X,\QQ)$. Then for any $i$ and $j$, we have
   \[  (\hat{\pi}^X_j)_\ast A_i^{hom}(X)_{\QQ}=0  \ \Rightarrow\  ({\Pi}^X_j)_\ast A_i^{hom}(X)_{\QQ}=0  \ .\]
    \end{lemma}
    
    \begin{proof} We have
      \[ (\Pi^X_j-\hat{\pi}^X_j)=0\ \ \ \hbox{in}\ H^{2m}(X\times X,\QQ)\ \]
      (where $m:=\dim X$). From Kimura's nilpotence theorem \cite{Kim}, it follows that there exists $N\in \NN$ such that
      \[ (\Pi^X_j-\hat{\pi}^X_j)^{\circ N}=0\ \ \ \hbox{in}\ A^{m}(X\times X,\QQ)\ .\]
      Developing this expression, we obtain
      \[  \Pi^X_j=(\Pi^X_j)^{\circ N}= P_1+ P_2+\cdots +P_m\ \ \ \hbox{in}\ A^{m}(X\times X)_{\QQ}\ ,\]  
      where each $P_j$ is a composition of correspondences containing at least one copy of $\hat{\pi}^X_j$. But then the right--hand side acts trivially on $A_i^{hom}(X)_{\QQ}$ (by hypothesis), and hence so does the left--hand side.    
    \end{proof}
    
     \end{proof}
    
    Let us now return to the Enriques variety $X=K/G$, and let us define cycles
      \[  \hat{\pi}_j^X:={1\over d} \, \Gamma_p\circ \Pi_j^K\circ {}^t \Gamma_p\ \ \ \in\ A^n(X\times X)_{\QQ}\ ,\]
      where $p\colon K\to X$ is the quotient morphism, and the $\Pi_j^K$ are as in theorem \ref{lin}.
      It follows from theorem \ref{lin} that
      \[  (\hat{\pi}_j^X)_\ast  A_0(X)_{\QQ}\ \subset\ (\Pi_j^K)_\ast A_0(K)_{\QQ}=0\ \ \ \hbox{for}\ j\ \hbox{odd}\ .\]
      In view of lemma \ref{jans2}, it follows that the vanishing (\ref{vanish}) holds for all odd $j$.

     It remains to establish the vanishing (\ref{vanish}) for even $j$. 
  The next lemma establishes two easy cases of (\ref{vanish}):
  
 \begin{lemma}\label{two} Set--up as in theorem \ref{main}. Then   
    \[  (\Pi^X_{j})_\ast A_0^{hom}(X)_{\QQ}=0\ \ \ \hbox{for}\ j\ge 4n-6.\]
    \end{lemma}

\begin{proof} (In view of lemma \ref{jans}, if the lemma is true for one Chow--K\"unneth decomposition, it is true for {\em all\/} Chow--K\"unneth decompositions.)

The case $j=4n-4$ is obvious (indeed, $\Pi^X_{4n-4}$ is just $X\times x$ for $x\in X$, and so the action factors over $A_0^{hom}(x)_{\QQ}=0$).
As for the second case, we observe that $H^2(X,\OO_X)=0$ so that $H^2(X,\QQ)$ is algebraic. By hard Lefschetz, $H^{4n-6}(X,\QQ)$ is also algebraic. This implies that the K\"unneth component $\pi^X_{4n-6}$ in cohomology is supported on $D\times C\subset X\times X$, where $D\subset X$ is a (possibly reducible) divisor and $C\subset X$ is a (possibly reducible) curve. The action of $\pi^X_{4n-6}$ on $A^j(X)_{\QQ}$ factors over $A^j(\wt{D})_{\QQ}$ (where $\wt{D}\to D$ is a desingularisation), hence $\pi^X_{4n-6}$ acts trivially $A_0(X)_{\QQ}$. Applying lemma \ref{jans2}, the same holds for any Chow--K\"unneth projector $\Pi^X_{4n-6}$.
\end{proof}

   We now state an equivariant version of corollaries \ref{corxu} and \ref{corxu2}:
   
 \begin{proposition}\label{equiv} Assumptions as in theorem \ref{main}. 
  
  \noindent
  (\rom1)
   There is a split injection
   \[ \Gamma_\ast\colon\ \ A_0(X)_{\QQ}= A_0(K)_{\QQ}^G\ \to\  A_0((A^\prime)^{(n)})_{\QQ}\ .\]
   
   \noindent
   (\rom2)
   For any $j$, there are split injections
   \[ (\Pi^X_j)_\ast A_0(X)_{\QQ}\ \to\ (\Pi_{j+4}^{(A^\prime)^{(n)}})_\ast A_0((A^\prime)^{(n)})_{\QQ}\ \]
   (here, $\Pi_j^X$ and $\Pi_{j}^{(A^\prime)^{(n)}}$ denote Chow--K\"unneth decompositions of $X$, resp. of $(A^\prime)^{(n)}$).   
   
    \end{proposition}  
   
  \begin{proof}
   To prove this, one needs to delve a bit into the proof of theorem \ref{xu}, i.e. one needs to understand Xu's result \cite{Xu}. 
   
   By construction of $K$, there is a commutative diagram (where vertical arrows are closed inclusions)
   \[  \begin{array}[c]{ccccc}
       K=K_n(A) &\to & K^{(n)}(A):=\sigma^{-1}(0) & \to &0\\
       \downarrow && \downarrow && \downarrow\\
       A^{[n]} &\to& A^{(n)} &\xrightarrow{\sigma} & \ \ A\ .\\
       \end{array}\]
   
   For any $\lambda\in P(n)$, let $K^\lambda:=\ker (s_\lambda)$ where
     \[ \begin{split} s_\lambda\colon\ \ A^{a_1}\times\cdots\times A^{a_r}\ &\to\ A\ ,\\
                  (x_1,\ldots,x_{\ell_\lambda})\ &\mapsto\ {\displaystyle\sum_{i=1}^{\ell_\lambda}} \lambda_i x_i\ .\\
                  \end{split}\]
   We have a stratification
     \[  K^\lambda=\coprod_{\tau\in A[e(\lambda)]} K^\lambda_\tau\ ,\]
     where
     \[ K^\lambda_\tau:= \bigl\{ (   x_1,\ldots,x_{\ell_\lambda})\in A^{\ell_\lambda}\ \vert\ \sum_{i=1}^{\ell_\lambda}  {\lambda_i  \over e(\lambda)} x_i=\tau\ \bigr\}\ .\]
     Let $S_a$ denote the symmetric group on $a$ elements.
     The action of $S_\lambda:=S_{a_1}\times\cdots\times S_{a_r}$ on $A^\lambda:=A^{a_1}\times\cdots\times A^{a_r}$   restricts to $K^\lambda$ (and to the $K^\lambda_\tau$); the quotient is denoted
     $K^{(\lambda)}$ (resp. $K^{(\lambda)}_\tau$).
     
     The natural morphism
     \[  A^{(\lambda)}:= A^\lambda/ S_\lambda\ \to\ A^{(n)} \]
     induces morphisms
     \[ K^{(\lambda)}_\tau\ \to\ K^{(n)}(A)\ .\]
     Then one defines correspondences
     \[    \hat{\Theta}^\lambda_\tau := ( K^{(\lambda)}_\tau \times_{K^{(n)}(A)} K)_{\rm red}\ \ \in\ A^\ast( K^{(\lambda)}_\tau\times K)\]
     (here $()_{\rm red}$ means one takes the subvariety with the reduced scheme structure), and
     \[ \Delta_{\lambda,\tau}:= {1\over d_{\lambda,\tau}}   \hat{\Theta}^\lambda_\tau\circ {}^t       \hat{\Theta}^\lambda_\tau\ \ \ \in A^{2n-2}(K\times K)_{\QQ} \]
     (where $d_{\lambda,\tau}\in\QQ$ is some constant).
     One can then prove (using the Beilinson--Bernstein--Deligne decomposition theorem) there is a decomposition
        \[ \Delta_K ={\displaystyle\sum_{\lambda\in P(n)} \sum_{\tau\in A[e(\lambda)]}}  \Delta_{\lambda,\tau}\ \ \ \hbox{in}\ A^{2n-2}(K\times K)_{\QQ} \]
     \cite[Lemma 2.5]{Xu}. This gives rise to an isomorphism of Chow motives
      \begin{equation}\label{isoK} h(K)\cong \bigoplus_{\lambda\in P(n)} \bigoplus_{\tau\in A[e(\lambda)]} h(K^{(\lambda)}_\tau)(n-\ell_\lambda)\ \ \ \hbox{in}\ \MM_{\rm rat} 
      \end{equation}
      \cite[Theorem 2.7]{Xu}.
      
      Next, one considers the natural morphism
      \[ \begin{split} \varphi\colon\ \ A\times K^{(\lambda)}_\tau\ &\ \to\ A^{(\lambda)}\ ,\\     
                                (x,z)\ &\ \mapsto\    t_x(z)\ \\
                                \end{split}\]
         (where $t_x$ is the translation by $x$), and one proves $\varphi$ induces an isomorphism of Chow motives
         \begin{equation}\label{iso2} \Gamma_\varphi\colon\ \ h(A\times K^{(\lambda)}_\tau)\cong h(A^{(\lambda)})\ \ \ \hbox{in}\ \MM_{\rm rat}\  \end{equation}
         \cite[Corollary 2.8]{Xu}. Combining isomorphisms (\ref{iso2}) and (\ref{isoK}) gives theorem \ref{xu}.     
         
      Consider now the Enriques variety
      \[ X=K/G\ ,\]
      where $G$ is a group acting freely on $K$ and induced by a finite order automorphism $\phi$.
      Since we are only interested in $0$--cycles, we only need to consider the one partition of length $n$, i.e. $\lambda=(1^n)$. We have that
      \[ K^{((1^n))}= K^{((1^n))}_0\]
      is invariant under the automorphism of $A^{(n)}$ induced by $\phi$; we write $K^{((1^n))}_0/G$ for the quotient. Fibre product gives rise to correspondences
      \[ \begin{split}  \hat{\Theta}^{(1^n)}_0(G):= ({K^{((1^n))}_0/ G}\times_{ {K^{(n)}(A)\over G}} X)_{\rm red}\ \ \ \in A^\ast(  K^{((1^n))}_0/G\times X)\ ,\\
                   \Delta_{(1^n),0}^G:= {1\over d_{(1^n),0}}\hat{\Theta}^\lambda_0(G)\circ {}^t \hat{\Theta}^\lambda_0(G)\ \ \ \in A^{2n-2}(X\times X)_{\QQ}\ .\\
                   \end{split}\]
          Taking $0$--cycles, we get a commutative diagram
          \[ \begin{array}[c]{cccc}
             A_0(K)_{\QQ}\cong &(\Delta_{(1^n),0})_\ast A_0(K)_{\QQ} &     \xrightarrow{  ({}^t \hat{\Theta}_0^{(1^n)})_\ast}& A_0(K_0^{((1^n))})_{\QQ}\\
            & \uparrow &&\uparrow\\
                A_0(X)_{\QQ}\cong &(\Delta^G_{(1^n),0})_\ast A_0(X)_{\QQ} &     \xrightarrow{  ({}^t \hat{\Theta}_0^{(1^n)}(G))_\ast}& A_0({K_0^{((1^n))}/ G})_{\QQ}\\  
                \end{array}\]     
           By (\ref{isoK}), the upper horizontal arrow is an isomorphism (with inverse given by $ ( \hat{\Theta}_0^{(1^n)})_\ast $). The vertical arrows are split injections. It follows that the lower horizontal arrow is an isomorphism (with inverse given by $ ( \hat{\Theta}_0^{(1^n)}(G))_\ast $).
           
           One checks that the morphism $\phi\in\hbox{Aut}(A)$ induces a morphism
           \[ \varphi^\prime\colon\ \ {A/ < \phi>}\times {K_0^{((1^n))}/ G}\ \to\ (A^\prime)^{(n)} \ .\]
           (Indeed, write $\phi= t\circ \phi_0$, where $t$ is a translation on $A$ and $\phi_0$ is a group automorphism. It is readily checked that $\phi_0$ commutes with $\varphi$, i.e.
           there is a commutative diagram
           \[  \begin{array}[c]{ccc}
              A\times K^{((1^n))}_0\ & \xrightarrow{\varphi}& A^{(n)}\\
           \ \  \ \ \downarrow{{}^{\phi_0\times \phi_0^{(n)}}}&&\ \  \downarrow{{}^{\phi_0^{(n)}}}  \\    
              A\times K^{((1^n))}_0\ & \xrightarrow{\varphi}& \ A^{(n)}\ ,\\
              \end{array}\]
 where $\phi_0^{(n)}$ is the morphism induced by $\phi_0$. As for the translation $t=t_a$, where $a\in A[n]$, we have a commutative diagram
   \[  \begin{array}[c]{ccc}
              A\times K^{((1^n))}_0\ & \xrightarrow{\varphi}& A^{(n)}\\
           \ \  \ \ \downarrow{{}^{t_{-a}\times t^{(n)}}}&&\ \  \downarrow{{}^{\hbox{id}}}  \\    
              A\times K^{((1^n))}_0\ & \xrightarrow{\varphi}& \ A^{(n)}\ .\\
              \end{array}\]  
       This proves the existence of $\varphi^\prime$.)

       Taking $0$--cycles, we get a commutative diagram
       \[ \begin{array}[c]{ccc} 
               A_0(A\times K^{((1^n))}_0)_{\QQ} & \xrightarrow{\varphi_\ast} & A_0(A^{(n)})_\QQ\\
               \uparrow &&\uparrow\\
               A_0(A/<\phi> \times {K^{((1^n))}_0/ G})_{\QQ} & \xrightarrow{(\varphi^\prime)_\ast} & \ \ A_0((A^\prime)^{(n)})_\QQ\ .\\   
               \end{array}\]
       By (\ref{iso2}), the upper horizontal arrow is an isomorphism (with inverse given by a multiple of $\varphi^\ast$). The vertical arrows are split injections. It follows that the lower horizontal arrow is an isomorphism. To prove (\rom1) of proposition \ref{equiv}, we consider the composition
       \[   A_0(X)_{\QQ}     \ \xrightarrow{  ({}^t \hat{\Theta}_0^{(1^n)}(G))_\ast}\ A_0({K_0^{((1^n))}/ G})_{\QQ} \ \xrightarrow{}\ 
              A_0(A/<\phi> \times {K^{((1^n))}_0/ G})_{\QQ} \ \xrightarrow{(\varphi^\prime)_\ast} \  A_0((A^\prime)^{(n)})_\QQ   \ ,\]
              where the first and last arrow are isomorphisms, and the second arrow (defined in the obvious way) is a split injection.    
       
         Statement (\rom2) of proposition \ref{equiv} is deduced from (\rom1) using finite--dimensionality; this is the same argument as corollary \ref{corxu2}.                             
           \end{proof} 
     
Using proposition \ref{equiv}, we can establish the required vanishing (\ref{vanish}) in some further cases:

\begin{lemma}\label{three} Set--up as in theorem \ref{main}. Then
    \[  (\Pi^X_{j})_\ast A_0^{hom}(X)_{\QQ}=0\ \ \ \hbox{for}\ j< 3n-4.\]
  Moreover, if $n= 4$ then 
    \[  (\Pi^X_{8})_\ast A_0^{hom}(X)_{\QQ}=0\ .\]
    \end{lemma}

\begin{proof} (Again, in view of lemma \ref{jans}, if the lemma is true for one Chow--K\"unneth decomposition, it is true for {\em all\/} Chow--K\"unneth decompositions.)

Thanks to proposition \ref{equiv}(\rom2), it suffices to prove
  \[  (\Pi_{k}^{(A^\prime)^{(n)}})_\ast A_0((A^\prime)^{(n)})_{\QQ}=0\ \ \ \hbox{for\ all\ }k< 3n\ .\]
  
 Let 
   \[\Delta_{A^\prime}=\Pi_0^{A^\prime}+\Pi_1^{A^\prime}+\Pi_2^{A^\prime}+\Pi_3^{A^\prime}+\Pi_4^{A^\prime}\ \ \ \hbox{in}\ A^2(A^\prime\times A^\prime)_{\QQ}\]
    be a Chow--K\"unneth decomposition for $A^\prime$. Since $p_g(A^\prime)=0$ and $A^\prime$ has finite--dimensional motive, we may suppose $\Pi_2^{A^\prime}$ is supported on $D\times D$, with $D\subset A^\prime$ a divisor (in other words, the ``transcendental part of the motive'' of $A^\prime$ is $0$, in the language of \cite{KMP}). Also we may suppose that $\Pi_0^{A^\prime}=x\times A^\prime$ for $x\in A^\prime$ and $\Pi_1^{A^\prime}$ is supported on $D\times A^\prime$, with $D\subset A^\prime$ a divisor (these are general facts, for the Chow--K\"unneth decomposition of any surface \cite{KMP}).
 
 As is well--known, the correspondences $\Pi_k^{(A^\prime)^{(n)}}$ are induced by correspondences
   \[  \Pi_k^{(A^\prime)^n}:={ \displaystyle\sum_{k_1+\cdots+k_n=k}}  \Pi_{k_1}^{A^\prime}\times\cdots\times \Pi_{k_n}^{A^\prime}\ \ \ \in A^{2n}((A^\prime)^n  \times (A^\prime)^n)\ ,\]
   which define an $S_n$--invariant Chow--K\"unneth decomposition of $(A^\prime)^n$.
   
 There is a commutative diagram
   \[ \begin{array}[c]{ccc}
           A_0((A^\prime)^n)_{\QQ}   & \xrightarrow{  (\Pi_k^{(A^\prime)^n})_\ast} &     A_0((A^\prime)^n)_{\QQ}\\
           \uparrow & & \uparrow\\
           A_0((A^\prime)^{(n)})_{\QQ}   & \xrightarrow{  (\Pi_k^{(A^\prime)^{(n)}})_\ast} &     A_0((A^\prime)^{(n)})_{\QQ}\\
           \end{array}\]

 Suppose now $j<3n$. Then each summand occurring in the definition of $\Pi_k^{(A^\prime)^n}$ contains at least one $\Pi_{k_\ell}^{A^\prime}$ with $k_\ell\le 2$. But this means (by the choice of $\Pi_\ast^{A^\prime}$ we have made above) that $\Pi_k^{(A^\prime)^n}$ is supported on $\hbox{(divisor)}\times (A^\prime)^n$ and so  acts trivially $0$--cycles:
   \[  (\Pi_{k}^{(A^\prime)^{(n)}})_\ast A_0((A^\prime)^{(n)})_{\QQ}=0\ \ \ \hbox{for\ all\ }k< 3n\ .\]
   
   It only remains to treat the case $n= 4$ and $k=12$. All the summands containing at least one $\Pi_{k_\ell}^{A^\prime}$ with $k_\ell\le 2$ act trivially on $0$--cycles (for the same reason as above). So we may suppose all the $k_\ell$ are $3$, and we need to prove that
   \[  \Bigl(  \Pi_3^{A^\prime}   \times \Pi_3^{A^\prime}\times  \Pi_3^{A^\prime} \times \Pi_3^{A^\prime}\Bigr) {}_\ast A_0 ((A^\prime)^{(4)})_{\QQ} =0\ .\]
  But there is a natural isomorphism
    \[  A_0 ((A^\prime)^{(4)})_{\QQ}\cong \bigl( {\displaystyle\sum_{\sigma\in S_4}} \Gamma_\sigma \bigr){}_\ast A_0 ((A^\prime)^{4})_{\QQ}    \ \subset\    A_0 ((A^\prime)^{4})_{\QQ}\ .\]
  One can check that the correspondences ${\displaystyle\sum_{\sigma\in S_4}} \Gamma_\sigma$ and 
  $ \Pi_3^{A^\prime}   \times \Pi_3^{A^\prime}\times  \Pi_3^{A^\prime} \times \Pi_3^{A^\prime}$ commute \cite[Lemma 3.4]{Kim}. It follows that
    \[  \Bigl(  \Pi_3^{A^\prime}   \times \Pi_3^{A^\prime}\times \Pi_3^{A^\prime}\times \Pi_3^{A^\prime}\Bigr) {}_\ast A_0 ((A^\prime)^{(4)})_{\QQ}\cong (\hbox{Sym}^4 \Pi_3^{A^\prime})_\ast 
        A_0 ((A^\prime)^{4})_{\QQ} \ ,\]
        where $\hbox{Sym}^4 \Pi_3^{A^\prime}$ is the projector defining the Chow motive $\hbox{Sym}^4 h^3(A^\prime)$ in the language of \cite[Definition 3.5]{Kim}. The action of the correspondence
        $\hbox{Sym}^4 \Pi_3^{A^\prime}$ on cohomology is projection to $\wedge^4 H^3(A^\prime,\QQ)$, which is one--dimensional
            (since $\dim H^3(A^\prime,\QQ)=4$ ) and consists of Hodge classes:
                          \[    \wedge^4 H^3(A^\prime,\QQ)\ \subset H^{12}((A^\prime)^4,\QQ) \cap F^6\ ,\]
                          where $F^\ast$ denotes the Hodge filtration.
           This implies that 
           \[         \hbox{Sym}^4 \Pi_3^{A^\prime}\ \in  \Bigl(  H^4((A^\prime)^4,\QQ) \cap F^2\Bigr) \otimes     \Bigl(  H^{12}((A^\prime)^4,\QQ) \cap F^6\Bigr)   \ \subset\ H^{16}((A^\prime)^4\times (A^\prime)^4),\QQ)\ .\]                          
         Next, we note that the Hodge conjecture is known to be true for self--products of abelian surfaces \cite[7.2.2]{Ab2}. This implies the same is true for the quotient variety $(A^\prime)^4$. (Indeed, let $p\colon A^4\to (A^\prime)^4$ denote the quotient morphism, and assume $a\in H^\ast((A^\prime)^4,\QQ)$ is a Hodge class. Then $p^\ast(a)$ is a cycle class. It follows that $p_\ast p^\ast(a)$ is a cycle class, and $p_\ast p^\ast(a)$ is a multiple of $a$ because of the isomorphism $H^i((A^\prime)^4,\QQ)\cong H_{8-i}((A^\prime)^4,\QQ)$.)
         
      Using the truth of the Hodge conjecture, we find that there is a cohomological equality
            \[  \hbox{Sym}^4 \Pi_3^{A^\prime}=      \gamma
             \ \ \ \hbox{in}\ H^{16}((A^\prime)^4\times (A^\prime)^4,\QQ)\ ,\]
             where $\gamma$ is a cycle supported on $V\times W$ for closed subvarieties $V, W\subset (A^\prime)^4$ of codimension $2$ resp. $6$.
       Since $(A^\prime)^4$ has finite--dimensional motive, this implies (proposition \ref{quotientnilp}) there exists $N\in \NN$ such that
         \[ \Bigl(\hbox{Sym}^4 \Pi_3^{A^\prime}-     \gamma \Bigr)^{\circ N}=0  \ \ \ \hbox{in}\ A^{8}((A^\prime)^4\times (A^\prime)^4)_{\QQ}\ .\]         
              Developing this expression (and noting that $\hbox{Sym}^4\Pi_3^{A^\prime}$ is idempotent), this gives a rational equivalence
   \[  \hbox{Sym}^4 \Pi_3^{A^\prime}= ( \hbox{Sym}^4 \Pi_3^{A^\prime})^{\circ N}=  S_1+\cdots +S_m         \ \ \ \hbox{in}\ A^{8}((A^\prime)^4\times (A^\prime)^4)_{\QQ}\ ,\]
   where each $S_i$ is a composition of correspondences in which $\gamma$ occurs at least once. But $\gamma$ acts trivially on $0$--cycles (for dimension reasons) and so the right--hand side also acts trivially on $0$--cycles, and we are done.            
         \end{proof}  
         
 Theorem \ref{main} can now be proven by combining lemmas \ref{two} and \ref{three}. Indeed, suppose $n=3$ or $n=4$. Then all even integers $j\in [0,4n-4]$ are either covered by lemma \ref{two} or covered by lemma \ref{three}. It follows there is no Chow--K\"unneth projector $\Pi_j^X$ acting non--trivially on $A_0^{hom}(X)_{\QQ}$ (i.e., the vanishing (\ref{vanish}) is proven), and so this group is trivial.        
   \end{proof}

\begin{remark}\label{examples} Clearly, theorem \ref{main} applies to the examples furnished by proposition \ref{ex}.
\end{remark}

\begin{remark} Note that the assumption on the group $G$ in theorem \ref{main} is more restrictive than just asking that $G$ is a group of natural automorphisms. Also, the dimension hypothesis $n\le 4$ was merely made for commodity, and is perhaps not really necessary. However, in view of the fact that all known examples of Enriques varieties dominated by generalized Kummer varieties (are given by proposition \ref{ex} and so) fit in with these hypotheses, it seems trifling to worry too much about these restrictions.
\end{remark}

As a corollary, some cases of the generalized Hodge conjecture are verified:

\begin{corollary}\label{BS} Let $X$ be an Enriques variety as in theorem \ref{main}. Then $H^j(X,\QQ)$ is supported on a divisor for all $j>0$.
\end{corollary}

\begin{proof} As is well--known \cite{BS}, this holds for any variety with trivial Chow group of zero--cycles.
\end{proof}

One can also say something about codimension $2$ cycles:

\begin{corollary} Let $X$ be an Enriques variety as in theorem \ref{main}. Then $A^2_{AJ}(X)_{\QQ}=0$.
\end{corollary}

\begin{proof} Again, this is true for any variety with trivial Chow group of zero--cycles \cite{BS}.
\end{proof}

\vskip1cm
\begin{nonumberingt} The ideas developed in this note came to fruition after the Strasbourg 2014---2015 groupe de travail based on the monograph \cite{Vo}. Thanks to all the participants of this groupe de travail for the stimulating atmosphere. Thanks to the referee for insightful comments.
Many thanks to Yasuyo, Kai and Len for lots of pleasant lunch breaks.
\end{nonumberingt}

\end{document}